\documentclass[12pt]{article}
\usepackage{amsmath,amsthm,amsfonts,amssymb}
\usepackage[utf8]{inputenc}
\usepackage{times}
\usepackage[T1]{fontenc}
\usepackage[a4paper]{geometry}
\def\C{{\mathbb C}}
\def\D{{\mathbb D}}
\def\N{{\mathbb N}}
\def\R{{\mathbb R}}

\def\S{{\mathbb S}}

\def\O{{\mathrm{O}}}
\def\U{{\mathrm{U}}}
\def\cC{{\mathcal{C}}}
\def\Crit{\Gamma_{\!\varphi}}
\def\Discr{\Delta_\varphi}
\DeclareMathOperator{\re}{re}
\DeclareMathOperator{\im}{im}
\DeclareMathOperator{\ind}{ind}
\DeclareMathOperator{\Ker}{Ker}
\def\etc{,\dots,}
\def\tu#1{\textup{#1}}
\def\tref#1{\textup{\ref{#1}}}

\def\sN{\mathrm{N}^\ast}
\def\sSN{\mathrm{SN}^\ast}
\def\sT{\mathrm{T}^\ast}
\def\sDT{\mathrm{DT}^\ast}
\def\sST{\mathrm{ST}^\ast}
\def\rT{\mathrm{T}}

\def\d{\mathrm{d}}
\def\dc{\d^{\scriptscriptstyle{\C}}}
\def\dlambda{\d\lambda}
\def\drho{\d\rho}

\def\df{\d f}
\def\dg{\d g}
\def\dh{\d h}
\def\dq{\d q}
\def\dx{\d x}
\def\dy{\d y}
\def\dz{\d z}

\def\dvf{\vec} 

\def\del{\partial}

\def\on{\colon}
\def\too{\longrightarrow}
\def\mapstoo{\longmapsto}
\def\into{\hookrightarrow}
\def\stem#1{\item[\textbf{#1)}]}
\def\itemup#1{\item[\textrm{#1)}]}
\def\ups{\upsilon}
\def\eps{\varepsilon}
\def\sgn{\epsilon}

\def\prg#1{\langle #1 \rangle}
\def\de{\mathrel{:=}}
\def\rst#1{{\upharpoonright}_{\!#1}}

\def\lrp#1{\left(#1\right)}

\def\with{\mathrel{:}}

\def\hook{\mathbin{\lrcorner}}
\theoremstyle{plain}
\newtheorem{theorem}{Theorem}
\newtheorem{extension-theorem}{Extension Theorem}

\newtheorem{proposition}[theorem]{Proposition}
\newtheorem{lemma}[theorem]{Lemma}
\newtheorem*{claim}{Claim}
\theoremstyle{definition}
\newtheorem{definition}[theorem]{Definition}
\newtheorem{example}[theorem]{Example}
\theoremstyle{remark}

\newtheorem*{remark*}{Remark}
\newtheorem{remark}[theorem]{Remark}

\def\hnu{\tilde\nu}
\def\cnu{\bar\nu}
\def\Lnu_#1{{\mathrm{L}_#1\nu}}

\begin{document}

\title{From Morse Functions to Lefschetz Fibrations\\ on Cotangent Bundles}

\author{Emmanuel \textsc{Giroux}\thanks
{\ CNRS and ENS--PSL (DMA, UMR 8553); e-mail: emmanuel.giroux@ens.fr.}}

\date{Paris, Summer 2025}

\maketitle

Although there exist many (closed integral) symplectic manifolds beyond complex
projective manifolds, it was demonstrated by S.~Donaldson that all of them admit
Lefschetz type pencils with symplectic fibers \cite{Do}, and that those could be
helpful to investigate the geometry. This set of ideas was further developed by
his school, especially by P.~Seidel who built on them to study Fukaya categories
\cite{Se}, with a slight change of framework: in place of closed symplectic
manifolds, he had to consider Liouville/Weinstein domains, and he consequently
replaced Lefschetz pencils by Lefschetz fibrations. The existence of symplectic
Lefschetz fibrations on Weinstein domains was then established in \cite{GP} by
adapting Donaldson's asymptotic methods. These fibrations are easy to define but
their geometry looks quite subtle. Actually, except maybe in dimension $4$, the
abundant literature on Lefschetz fibrations describes rather few significant
concrete examples, and the proof of the general existence result mentioned above
is not very instructive. The goal of this elementary paper is to produce and
analyze explicit Lefschetz fibrations on cotangent bundles:

\begin{extension-theorem}[for closed manifolds] \label{t:closed}
Let $M$ be a closed manifold, $\varphi \on M \to \R$ a Morse function, and $\nu$
an adapted gradient of $\varphi$ which satisfies the Morse--Smale condition.
Then $\varphi$ extends to a \tu(homotopically unique\tu) symplectic Lefschetz
fibration $h = f+ig \on \sT M \to \C$ whose imaginary part is the function
\[ g \on \sT M \to \R, \quad (p,q) \mapsto g(p,q) = \prg{p, \nu(q)} \]
and whose real part $f$ is $1$--homogeneous near infinity. In addition, $h$ can
be chosen equivariant under the actions of the fiberwise antipodal involution
and the complex conjugation.
\end{extension-theorem}

This statement requires some clarifications:
\begin{itemize}
\item
throughout the paper, $M$ is identified with the zero-section of $\sT M$, and
$h$ extends $\varphi$ in the sense that $h \rst M = \varphi$;
\item
a function $f \on \sT M \to \R$ is $d$--homogeneous near infinity if there is a
compact set $W_0 \subset \sT M$ such that $f(tp,q) = t^d f(p,q)$ for all $(p,q)
\in \sT M - W_0$ and all $t \ge 1$; 
\item
a symplectic Lefschetz fibration is a map satisfying the axioms of Definition
\ref{d:slf};
\item
a vector field $\nu$ is an adapted gradient of $\varphi$ if $\nu \cdot \varphi
> 0$ away from the critical points, and near each critical point $a$, there are
local coordinates $(q_1 \etc q_n)$ centered on $a$ (called Morse coordinates) in
which 
\begin{align*}
	\varphi(q) &= \varphi(a) + \sum_1^n \sgn_j q_j^2, \quad \sgn_j \in \{-1,1\},
	\\ \llap{\text{and}\quad}
	\nu(q) &= 2 \sum_1^n \sgn_j q_j \del_{q_j}.
\end{align*}
\end{itemize}
Most likely Theorem \ref{t:closed} still holds for an arbitrary gradient $\nu$
(however, choosing an adapted gradient leads to simpler calculations and a nicer
overall picture). In contrast, Theorem \ref{t:closed} fails if the (adapted)
gradient $\nu$ violates the Morse--Smale condition (see Remark \ref{r:MS}).

\medskip

A wellknown instance of a Morse function which extends to a simple Lefschetz
fibration is the following:

\begin{example}[the sphere case]
Let $M = \S^n$ denote the unit sphere
\[ M = \left\{x = (x_1 \etc x_{n+1}) \in \R^{n+1} \with
   \sum_1^{n+1} x_j^2 = 1\right\} \]
and consider the Morse function $\varphi \on M \to \R$ given by the coordinate
$x_{n+1}$ restricted to $M$. Then (as explained in Example \ref{x:locmod}) the
cotangent bundle $\sT M$ is symplectomorphic to the complex affine quadric 
\[ W = \left\{z = x+iy = (z_1 \etc z_{n+1}) \in \C^{n+1} \with
   \sum_1^{n+1} z_j^2 = 1\right\} \]
(with the symplectic form induced by the standard Kähler form of $\C^{n+1}$),
and the restriction of the coordinate $z_{n+1}$ to $W$ is a holomorphic (hence
symplectic) Lefschetz fibration $h$ extending $\varphi$. Its regular fiber is
the cotangent bundle of $\S^{n-1}$. This fibration does not quite satisfy the 
homogeneity condition of Theorem \ref{t:closed} but its real and imaginary parts
have the same growth on the cotangent fibers, which is the main point to have a
correct behavior near infinity.
\end{example}      

Theorem \ref{t:closed} can be easily generalized to exhausting Morse functions
on non-compact manifolds; the construction is exactly the same but the Weinstein
manifolds we obtain as regular fibers of the Lefschetz extension are no longer
of finite type. There is also a version for Morse functions on cobordisms which
we state below. By convention, a Morse function $\varphi \on M \to \R$, where
$M$ is a compact manifold with boundary, is a function whose critical points are
non-degenerate and which is locally constant and regular along $\del M$. We then
denote by $\del^-M$ (resp.\ $\del^+M$) the union of the boundary components
where the gradient $\nu = \nabla\varphi$ is pointing inward (resp.\ outward) and
we regard $M$ as a cobordism from $\del^-M$ to $\del^+M$.

\begin{extension-theorem}[for cobordisms] \label{t:compact}
Let $M$ be a compact manifold with boundary, $\varphi \on M \to \R$ a Morse
function, and $\nu$ an adapted gradient of $\varphi$ which satisfies the
Morse--Smale condition. Then $\varphi$ extends to a \tu(homotopically unique\tu)
map $h = f+ig \on \sT M \to \C$ with the following properties, where $B \de
\varphi(M) \oplus i\R \subset \C$\tu:
\begin{itemize}
\item
$g(p,q) = \prg{p, \nu(q)}$ for all $(p,q) \in \sT M$, and $f$ is $1$--homogeneous near infinity;
\item
$h \rst{h^{-1}(B)} \on h^{-1}(B) \to B$ is a symplectic Lefschetz fibration\tu;
\item
$\sT M$ retracts onto $h^{-1}(B)$ along the orbits of the Hamiltonian field of
$g$.
\end{itemize}
\end{extension-theorem}

Here are a few geometric properties of $h$ which are direct consequences of the
statements of Theorems \ref{t:closed} and \ref{t:compact}:
\begin{itemize}
\item
The critical points of $h$ lie in $M \subset \sT M$ and coincide with those of
$\varphi$.
\item
For every critical point $a$, the Lefschetz thimble of $h$ over the real path
reaching $\varphi(a)$ from below (resp.\ from above) is the conormal bundle of
the stable (resp.\ unstable) disk of $a$ for $\nu$ (cf.\ Lemma \ref{l:hnu}).
\item
For any regular fiber $F$ of $h$, the composite map $F \into \sT M \to M$ is
$(n-1)$--connected, where $n \de \dim M$ (the reason is that $\sT M \sim M$ is
homotopically obtained from $F \times \D^2$ by attaching $n$--handles).
\item
For every regular value $u$ of $\varphi$, the antipodal involution of $\sT M$
preserves the Lefschetz fiber $F_u \de h^{-1}(u)$ and reverses its symplectic
form. Thus, it defines a real structure on $F_u$ whose real locus is the regular
level set $Q_u \de \{\varphi=u\} \subset M$. Hence, this real structure changes 
drastically when $u$ crosses a critical value of $\varphi$. We will also show
that $F_u$ is a Weinstein manifold for the $1$--form induced by the canonical
Liouville form of $\sT M$, and that the homotopy class of this Weinstein
structure does not depend on the real regular value $u$ (cf.\ Proposition
\ref{p:w}).
\end{itemize}

As a consequence of Theorem \ref{t:closed} and Proposition \ref{p:w}, to every
``upgraded Morse function'' $(\varphi,\nu)$ on a closed manifold $M^n$, one can
canonically associate a Weinstein manifold $F^{2n-2}$ of finite type, its
``Lefschetz fiber'', which is the regular fiber of the (homotopically unique)
symplectic Lefschetz fibration $h$ extending $\varphi$. The geometry of $F$ is
quite interesting; it contains the vanishing cycles of the critical points along
with the regular level sets of $\varphi$ as exact Lagrangian submanifolds, and
it can be pretty explicitly described from those objects and their incidence
relations (see Section \ref{s:lf}). In short, Morse theory says that, when the
real parameter $u$ passes through a critical value, the associated regular level
set $Q_u \de \{\varphi=u\}$ undergoes a surgery of some index $k$, which means
that a copy of $\S^{k-1} \times \D^{n-k}$ is removed and replaced with a copy
of $\D^k \times \S^{n-k-1}$. As we will see, these two copies actually live
together in the Lefschetz fiber $F_u \de h^{-1}(u)$ where they form an embedded
Lagrangian sphere $\S^{n-1} = (\S^{k-1} \times \D^{n-k}) \cup (\D^k \times
\S^{n-k-1})$ which is the vanishing cycle of the corresponding critical point of
$h_\varphi$. Furthermore, each vanishing cycle comes tagged with a Morse index.
Thus, the symplectic invariants of $F$ can be regarded as invariants of the
pair $(\varphi,\nu)$. 

The construction of the Lefschetz fibration $h$ extending $\varphi$ roughly goes
as follows: by a coarse complexification process, we first extend $\varphi$ to 
an approximately holomorphic map $h^0 \on W_\delta \to \C$ on the small closed
$\delta$--tube $W_\delta$ about the zero-section in $\sT M$. The critical points
of $h^0$ have the required shape and its fibers are symplectic submanifolds away
from critical points. This map, however, is not a fibration at all (most fibers
over real values are cotangent tubes about the corresponding level sets of
$\varphi$), and we need to extend it over larger tubes in $\sT M$ in order to
complete the fibers till all of them have the same topology. It turns out that
this can be achieved by a very simple trick, namely, a convenient reordering of
the Morse--Bott function $(\re h^0) \rst{ \del W_\delta }$.

\subsection*{Credits and methods} 

This work is tightly related to the work of J.~Johns \cite{Jo1}, and so a few
comments are in order. In his PhD thesis, Johns obtained a weak version of the
extension result stated above (see \cite[Theorem A]{Jo1}): for any self-indexing
Morse function $\varphi \on M \to \R$ whose critical indices, besides $0$ and
$n \de \dim M$, lie in the interval $[(n-1)/2,(n+1)/2]$, he built a Weinstein
manifold $W$, a symplectic Lefschetz fibration $h \on W \to \C$ and an exact
Lagrangian embedding $\iota \on M \to W$ which contains all critical points of
$h$ and essentially satisfies $h \circ \iota = \varphi$. He sketched also a
proof that $\iota$ should be a homotopy equivalence, but he was not able to
identify $W$ with $\sT M$. Still, Theorem~A of \cite{Jo1} yields a Lefschetz
fibration which is definitely similar to ours, and in \cite{Jo2} Johns used it
to compare the flow category of $\varphi$ with the directed Donaldson--Fukaya
category of $h \on W \to \C$.

More recently, S.~Lee \cite{Le} also proposed an algorithm producing a Lefschetz
fibration on the disk cotangent bundle of a closed manifold $M$ out of a handle
decomposition of~$M$. Presumably his construction is roughly equivalent to ours, 
although his solution to the main problem encountered by Johns remains unclear
(cf.\ \cite[Subsection 6.3]{Le}). 

Our approach in this paper is much more direct than those of Johns and Lee, and
it is originated in the study of contact convexity. Example I.4.8 in \cite{Gi}
shows that the canonical contact structure on the sphere cotangent bundle of any
closed manifold is ``convex'' in the sense of Eliashberg--Gromov \cite{EG},
which means that it is invariant under some gradient flow. Years later, I
realized that this condition is equivalent to the existence of a ``supporting
open book'', and the Lefschetz fibration we construct here is a natural filling
of this open book.

\subsection*{Acknowledgements}

I wish to thank Paul Biran, Octav Cornea and Matija Sreckovi\'c for their
interest in this work and their many comments on a preliminary version of this
paper.

\section{Lefschetz fibrations and their local behavior} \label{s:slf}

Let $W$ be a manifold given with a Liouville form $\lambda$, namely a $1$--form
whose differential $\omega \de \dlambda$ is symplectic; the associated Liouville
vector field $\dvf\lambda$ is defined by $\dvf\lambda \hook \omega = \lambda$. 
The pair $(W,\lambda)$ is called:
\begin{itemize}
\item
a Liouville domain if $W$ is compact and $\lambda$ induces a positive contact
form on $\del W$ oriented as the boundary of $(W,\omega)$ (the latter condition
is equivalent to $\dvf\lambda$ pointing transversely outward along $\del W$);
\item
a Liouville manifold if $W$ is exhausted by Liouville domains $(W_k, \lambda
\rst{ W_k })$, $k \in \N$, and if the Liouville field $\dvf\lambda$ is complete;
\item
a Weinstein manifold if it is a Liouville manifold and if the Liouville field
$\dvf\lambda$ is gradientlike for some (unspecified but homotopically unique)
exhausting function; in this case, the form $\lambda$ is called a Weinstein
structure.
\end{itemize}
As explained in \cite{Ci}, the term gradientlike here has the following meaning:
$\dvf\lambda$ is gradientlike (for a function $\rho \on W \to \R$) if there is 
an isomorphism $L \on \rT W \to \sT W$ which is positive ($\prg{ L\eta, \eta }
> 0$ for every nonzero tangent vector $\eta$) and which sends $\dvf\lambda$ to
the differential of some function (namely, $\drho$); alternatively, $\rho$ is
called a Lyapunov function for $\dvf\lambda$.

The cotangent bundle $\sT M$ of the manifold $M$, endowed with its canonical
Liouville form $\lambda$ and symplectic structure $\omega \de \dlambda$, is a
Weinstein manifold: the (fiberwise radial) Liouville field $\dvf\lambda$ is
gradientlike for the Kinetic energy $\rho$ of any Riemannian metric on $M$. Our
convention is that, if $(q_1 \etc q_n)$ are local coordinates on $M$ and $(p_1
\etc p_n)$ denote the associated cotangent coordinates, then
\[ \lambda = \sum_1^n p_j \dq_j, \quad \text{and} \quad
   \omega = \sum_1^n \d p_j \wedge \dq_j. \]
Symplectic Lefschetz fibrations were first introduced by S.~Donaldson \cite{Do}
together with symplectic Lefschetz pencils on closed manifolds. Nowadays, there
is a specific notion of symplectic Lefschetz fibration attached to each class of
symplectic manifolds (such as closed symplectic manifolds, Liouville manifolds,
\ldots), the basic extra requirement being that the regular fiber lies in the
same class as the total space. Here is the notion we use in this paper:

\begin{definition}[Lefschetz fibrations on Liouville manifolds] \label{d:slf}
Let $(W,\lambda)$ be a Liouville manifold of dimension $2n$. A map $h \on W \to
\C$ is a \emph{symplectic Lefschetz fibration} if the following properties hold:
\begin{itemize}
\itemup{1}
the critical points of $h$ are of complex Morse type: each of them is the center
of complex coordinates $(z_1 \etc z_n)$ in which $\omega=\dlambda$ is a positive
$(1,1)$--form at $0$ and $h(z) = h(0) + \sum_1^n z_j^2$ (this model is actually
more restrictive than necessary, but it will easily be achieved);
\itemup{2}
the distribution $\Ker\dh \subset \rT W$ consists of symplectic subspaces (of
corank $2$ except at critical points), and the singular connection formed by its
symplectic orthogonal complement is complete: parallel transport does not escape
to infinity in finite time, but it does crash some points to the critical points
of $h$ (in a way prescribed by the quadratic local model); 
\itemup{3}
the manifold $W$ is exhausted by Liouville domains $(W_k, \lambda \rst{ W_k })$
such that, for every $w \in \C$ and for all sufficiently large $k \ge k_w$, the
fiber $F_w \de h^{-1}(w)$ intersects $\del W_k$ transversely along a positive
contact submanifold of $\del W_k$, and in addition the Liouville field on $F_w$
dual to $\lambda \rst{ F_w }$ is complete.
\end{itemize}
\end{definition}

The latter axiom above ensures that every regular fiber of $h$ is a Liouville
manifold. Moreover, because the Liouville forms of the fibers are induced by the
global $1$--form $\lambda$ of $W$, the holonomy maps given by parallel transport
along the (complete) connection are exact symplectomorphisms; as a result, there
is a consistent notion of exact Lagrangian submanifolds in the fibers $F_w$.
   
\begin{remark}[the case of cobordisms]
In our extension theorem \ref{t:compact} for cobordisms, the set $h^{-1}(B)$ is
not a genuine Liouville manifold but $h \rst{ h^{-1}(B) } \on h^{-1}(B) \to \C$
is a Lefschetz fibration in the sense that it satisfies the three axioms of the
above definition.
\end{remark}

\begin{example}[the local model under various angles] \label{x:locmod}
Consider $\C^n$ with its standard symplectic form $\omega \de \sum_1^n
\dx_j \wedge \dy_j$, where the complex coordinates are $z_j = x_j+iy_j$, $1 \le
j \le n$. This is a Weinstein manifold (actually, a Stein manifold); indeed,
$\omega = \dlambda$ where $\lambda \de \frac12 \sum_1^n (x_j\dy_j - y_j\dx_j)$
and $\lambda = \dc(|z|^2/4)$, so the Liouville field $\dvf\lambda$ is the gradient
of the function $z \mapsto |z|^2/4$. (Our convention in this paper is that
$\dc\rho(\_) = - \drho(i\_)$ for any function $\rho \on \C^n \to \R$.)

With the above definition, the quadratic function
\[ h \on \C^n \too \C, \quad
   z = (z_1 \etc z_n) \mapstoo h(z) \de \sum_1^n z_j^2, \]
is a symplectic Lefschetz fibration. For every $w \in \C$, the fiber $F_w \de
h^{-1}(w)$ is the complex affine quadric
\[ F_w = \left\{z \in \C^n \with \sum_1^n z_j^2 = w\right\}. \]
It has a nodal singularity at the origin if $w=0$. Otherwise, $F_w$ is smooth
and symplectomorphic to the cotangent bundle of the sphere $\S^{n-1}$. Indeed,
each rotation $z \mapsto e^{i\alpha} z$ preserves $\omega$ and takes $F_w$ to
$F_{e^{2i\alpha} w}$, and if $u \de e^{2i\alpha} w$ is a positive real number,
the map
\[ F_u \too \sT\S^{n-1}, \quad 
   z = x+iy \mapstoo (p,q) = (-|x|\,y, x/|x|), \]
is a symplectomorphism which pulls back the canonical Liouville form of $\sT
\S^{n-1}$ to the $1$--form induced by the (rotationally invariant) Liouville
form $\lambda$. The $(n-1)$--sphere
\[ Z_u \de F_u \cap \R^n = \left\{x \in \R^n \with
   \sum_1^n x_j^2 = u\right\}, \]
collapses to $0$ as $u \to 0$ and is called the vanishing cycle of $F_u$. Its
inverse image under the rotation $z \mapsto e^{i\alpha}z$ is the vanishing cycle
$Z_w$ in $F_w$; it can be characterized as the minimum locus of the function
$|Z|^2$ on $F_w$. 

The map $h$ will play a crucial role in our construction since it provides the
extension we want for any Morse function near a critical point. In the rest of
this section, we review some important geometric properties of $h$; though these
considerations do not formally enter in the proofs of Theorems \ref{t:closed}
and \ref{t:compact}, they are helpful to apprehend the geometry of the Lefschetz
fiber we will obtain.

\medskip

\stem{a} (Recollections on parallel transport.)
We briefly tell here how to determine the parallel transport between fibers of
$h$. The connection being the symplectic orthogonal complement of $\Ker \dh$,
it is the complex line field $z \mapsto \C \bar z$ on $\C^n - \{0\}$. Therefore,
given any real vector subspace $P \subset \R^n$, the complex subspace $\C P$ is
preserved by parallel transport since it is invariant under complex conjugation.
Moreover, the parallel transport in $\C P$ has the same behavior for all $P$ of
any fixed dimension because the map $h$ and the symplectic form $\omega$ (hence
the connection) are invariant under the action of the orthogonal group $\O_n
\subset \U_n$, which also acts transitively on the grassmannian. Finally, the
subspaces $\C P$ with $\dim P = 2$ cover $\C^n$ (every point $z \in \C^n$ is in
the complex span of $z + \bar z$ and $i (z - \bar z)$), so it suffices to study
parallel transport in a complex plane~$\C P$. Actually, since every fiber $F_w$
of $h$ meets such a $\C P$ along a copy of $\sT\S^1$, the parallel transport
between any two fibers can be described as a family (parameterized by $P$) of
annulus diffeomorphisms. (This reflects the $(\sT\S^1 - \S^1)$--bundle structure
of $\sT\S^{n-1} - \S^{n-1}$ over the grassmannian of planes in $\R^n$.) 

We now focus on parallel transport over (arcs of) circles about the origin in
$\C$, which is generated by the Hamiltonian field of the function $|h|^2$. Given
a plane $P \subset \R^n$, we can find (appropriate and temporary) coordinates
$(z_1,z_2)$ on $\C P$ (which are actually unitary up to a factor $\sqrt2$) in
which $h = h \rst{ \C P }$ takes the form
\[ h(z_1,z_2) = z_1z_2. \] 
Then, setting $z_j = x_j+iy_j$, $j \in \{1,2\}$, the Hamiltonian field of $|h|^2
= |z_1|^2 |z_2|^2$ reads 
\[ 2(x_2^2+y_2^2) (x_1 \del_{y_1} - y_1 \del_{x_1}) 
   + 2(x_1^2+y_1^2) (x_2 \del_{y_2} - y_2 \del_{x_2}). \]
Observing that the functions $|z_1|^2$ and $|z_2|^2$ are first integrals of this
vector field, we see that its flow is given by
\[ (z_1,z_2, t) \in \C P \times \R \mapstoo
   \left(e^{2i\,|z_2|^2t} z_1, e^{2i\,|z_1|^2t} z_2\right) \in \C P. \] 
Thus, every solution $t \in \R \mapsto (z_1(t),z_2(t))$ with initial condition
$(z_1,z_2)$ at $t=0$ satisfies
\[ h \bigl(z_1(t),z_2(t)\bigr) = z_1(t) z_2(t)
   = e^{2i(|z_1|^2+|z_2|^2)t} z_1z_2  \quad \text{for all $t \in \R$.} \]
Hence, the time necessary for $h(z_1,z_2)$ to rotate by a angle $\alpha$ is $t
= \alpha/2(|z_1|^2+|z_2|^2)$. Combining this with the expression of the flow, we
get an explicit formula for the parallel transport $\tau_\alpha$ from a fiber
$F_w \cap \C P$ to the fiber $F_{e^{i\alpha}w} \cap \C P$:
\[ \tau_\alpha (z_1,z_2) = \left(e^{is\alpha} z_1, e^{i(1-s)\alpha} z_2\right),
   \quad \text{where $s \de |z_2|^2 / (|z_1|^2+|z_2|^2)$.} \]
We can also parameterize $F_w \cap \C P$ by the map
\[ \psi_w \on \C^* \too \C P, \quad
   z \mapsto (z_1,z_2) =
   \left(\frac z { |w|^{1/2} }, w \cdot \frac{ |w|^{1/2} } z\right), \]  
(normalized so as to take the unit circle to $Z_w \cap \C P$). Then, writing $z
= re^{i\theta}$, we obtain
\[ \psi_{e^{i\alpha}w}^{-1} \circ \tau_\alpha \circ \psi_w (r,\theta)
   = \left(r, \theta + \frac{\alpha}{1+r^4}\right). \]
For $\alpha = 2\pi$, we recover the fact that the monodromy is a right-handed
Dehn twist.

We henceforth reset our coordinates $z_j = x_j+iy_j$.

\smallskip

\stem{b} (The real forms of $h$.)
For $0 \le k \le n$, let $M_k \subset \C^n$ denote the Lagrangian plane spanned
by the coordinates $(y_1 \etc y_k, x_{k+1} \etc x_n)$, namely 
\[ M_k \de \{z = x+iy \in \C^n \with
   x_1 =\dots= x_k = y_{k+1} =\dots= y_n = 0\}. \]   
On $M_k$, the map $h$ is real-valued and $\varphi_k \de h \rst{ M_k } \on M_k
\to \R$ is the function
\[ \varphi_k (y_1 \etc y_k, x_{k+1} \etc x_n) 
   = - y_1^2 -\dots- y_k^2 + x_{k+1}^2 +\dots+ x_n^2, \]   
which is the standard model for a Morse function near a critical point of index
$k$. Accordingly, the level sets of $\varphi_k$,
\[ Q_{k,u} \de F_u \cap M_k = \{\varphi_k=u\} \subset M_k, \quad
   u \in \R, \ 0 \le k \le n, \]
represent the various real forms of the complex quadric $F_u$.

We (symplectically) identify $\C^n$ with $\sT M_k$ using the map
\[ z = (x'+iy', x''+iy'') \in \C^k \times \C^{n-1}
   \mapstoo (p,q) = \bigl((x',-y''), (y',x'')\bigr) \in \sT M_k. \]
In the coordinates $(p,q)$, the function $g(z) \de \im h(z) = 2 \sum_1^n x_jy_j$
takes the form $g(p,q) = \prg{p, \nabla\varphi_k(q)}$, as required in Theorems
\ref{t:closed} and \ref{t:compact}. On the other hand, in the coordinates $z_j
= x_j+iy_j$, the canonical $1$--form $\lambda_k$ of $\sT M_k$ reads  
\[ \lambda_k \de \sum_1^k x_j \dy_j - \sum_{k+1}^n y_j \dx_j = \dc\rho_k \]
where $\rho_k(z) \de \tfrac12 \left(\sum_1^k x_j^2 + \sum_{k+1}^n y_j^2\right)$.

It follows that the form $\lambda_{k,w} \de \lambda_k \rst{ F_w }$ has a dual
field $\dvf\lambda_{k,w}$ which is the gradient of $\rho_{k,w} \de \rho_k \rst{
F_w }$. Moreover, for any $w \in \C^*$, a calculation shows that $\rho_{k,w}$ is
a Morse--Bott function whose critical locus is the union of the following sets:
\begin{itemize}
\item
the $(k-1)$--sphere $Z_w \cap (\C^k \times \{0\})$;
\item
the $(n-k-1)$--sphere $Z_w \cap (\{0\} \times \C^{n-k})$;
\item
the intersection $F_w \cap M_k$, which is empty as soon as $w \notin \R$.
\end{itemize}
It is worth noticing that the function $\rho_{k,w}$ is not proper if $1 \le k
\le n-1$, so $\lambda_{k,w}$ is not a Weinstein structure. Still, the behavior
of $\rho_{k,w}$ near the vanishing cycle $Z_w$, especially when $w=u \in \R$,
is enlightening to analyze the bifurcation of the Weinstein structure which
occurs at a critical point (see Lemma \ref{l:sharp}).

\smallskip

\stem{c} (The vanishing cycle and the level sets of $\varphi_k$.)
We now assume that $w=u \in \R$. Then the fiber $F_u$ contains two obvious exact
Lagrangian submanifolds:
\begin{itemize}
\item
the level set $Q_{k,u} = \{\varphi_k=u\} = F_u \cap M_k$, which is diffeomorphic 
to $\S^{k-1} \times \R^{n-k}$ if $u<0$ and to $\R^k \times \S^{n-k-1}$ if $u>0$; 
\item
the vanishing cycle $Z_u$, which is invariant under the flow of the gradient
$\dvf\lambda_{k,u}$ of the Morse--Bott function $\rho_{k,u}$ (indeed, $\lambda_k
\rst{ Z_u } = 0$).
\end{itemize}
These two submanifolds intersect cleanly, along $Z_u \cap (\C^k \times \{0\})
\simeq \S^{k-1}$ if $u<0$ and along $Z_u \cap (\{0\} \times \C^{n-k}) \simeq
\S^{n-k-1}$ if $u>0$, the intersection being the attaching sphere in $Q_{k,u}$.
To complete this picture, we have the following result (see also \cite[Theorem
1.8.4]{Sr}) :
\end{example}

\begin{lemma}[parallel transport and Lagrangian surgery] \label{l:surgery}
For $u>0$, the parallel transport $\tau_\pi \on F_{-u} \to F_u$ maps the level
set $Q_{k,-u}$ to an exact Lagrangian submanifold of $F_u$ which is isotopic to
that obtained from $Q_{k,u}$ and $Z_u$ by \tu(the Morse--Bott version of the
right-handed\tu) Lagrangian surgery.    
\end{lemma}

Actually, the parameters involved in the Lagrangian surgery can be chosen so as
to produce a Lagrangian submanifold which is Hamiltonian isotopic to $\tau_\pi
(Q_{k,-u})$.  

\begin{proof}
We resume here the method and the notations used in Example \ref{x:locmod}-a
to study parallel transport.

We fix a plane $P \subset \R^n$ and use $\psi_w$ to identify $F_{-u} \cap \C P$
with $\C^*$. The complexified plane $\C P$ intersects $Z_{\pm u}$ in the circle
$\psi_{\pm u}(\S^1)$, while $\C P$ meets $Q_{k,s\pm u} - Z_{\pm u}$ if and only
if $P = P' \oplus P''$ where $P'$ and $P''$ are lines in $\R^k \times \{0\}$ and
$\{0\} \times \R^{n-k}$, respectively, and in this case, $Q_{k,\pm u} \cap \C P$
consists of the two rays $\psi_w (\{\theta=0\} \cup \{\theta=\pi\})$. The formula
\[ \psi_u^{-1} \circ \tau_\pi \circ \psi_{-u} (r,\theta)
   = \left(r, \theta + \frac{\pi}{1+r^4}\right), \]
and a little drawing then shows that $\tau_\pi$ sends $Q_{k,-u} \cap \C p$ to a
curve in $F_u \cap \C P$ which, up to isotopy, is obtained by the right-handed
surgery of $Q_{k,u} \cap \C P$ and $Z_u \cap \C P$.

These observations, applied to all planes $P$, prove the lemma.
\end{proof}

\section{The two lifts of a vector field} \label{s:lifts} 

An important basic fact in our construction is that every vector field $\nu$ on
$M$ has two natural lifts:  
\begin{itemize}
\item
a vector field $\hnu$ on $\sT M$ which preserves the canonical Liouville form
$\lambda$; this property $\hnu \cdot \lambda = 0$ and the Cartan formula for the
Lie derivative then imply that $\hnu$ is Hamiltonian, with Hamiltonian function
$(\hnu \hook \lambda)(p,q) = - \prg{p, \nu(q)}$;
\item
a vector field $\cnu$ on $\sST M$ which  preserves the contact structure defined
by $\lambda$; actually, $\cnu$ is just the image of $\hnu$ under the projection
$\sT M - M \to \sST M$, where $\sT M - M$ can be viewed as the symplectization
of $\sST M$.
\end{itemize}
The next two lemmas describe elementary properties of these two lifts. 

\begin{lemma}[properties of $\hnu$] \label{l:hnu}
Let $\nu$ be a vector field on $M$ with nondegenerate singularities, and let
$\hnu$ denote its Hamiltonian lift on $\sT M$.

\stem{a}
The singularities of $\hnu$ lie on the zero-section $M$, coincide with those of
$\nu$, and are nondegenerate. Actually, $\hnu$ is tangent to each fiber $\sT_aM$
over a zero $a$ of $\nu$, and on $\sT_aM$, it agrees with the negative transpose
of the linearization $\Lnu_a$ of $\nu$ at~$a$.

\stem{b}
A singularity $a$ is hyperbolic for $\nu$ if and only if it is for $\hnu$, and
its stable \tu(resp.\ unstable\tu) manifold for $\hnu$ is the conormal bundle of
its stable \tu(resp.\ unstable\tu) manifold for~$\nu$.
\end{lemma}

\begin{proof}
Everything is obvious except maybe the assertion about the stable and unstable
manifolds of~$a$ for~$\hnu$ when $a$ is a hyperbolic singularity of~$\nu$. The
conormal bundle of an invariant manifold for~$\nu$ is an invariant manifold for
$\hnu$ but we have to show that the conormal bundle $\sN E^-(a)$ of the stable
manifold $E^-(a)$ of $a$ for $\nu$ is the stable manifold of $a$ for $\hnu$.

If $a$ is a hyperbolic singularity of $\nu$ then $\sT_aM$ splits as the direct
sum of a stable subspace $P_a^-$ and an unstable subspace $P_a^+$ for $\hnu \rst
{ \sT_aM } = -\Lnu_a^\top$. Thus $a$ is a hyperbolic zero of $\hnu$. Next, the
tangent space of $\sN E^-(a)$ at $a$ is the direct sum of $\rT_aE^-(a) \subset
\rT_aM$ and the conormal of this subspace in $\sT_aM$, which is $P_a^-$. This
shows that $\hnu$ is contracting on $\rT_a \bigl(\sN E^-(a)\bigr)$ and implies
that $\sN E^-(a)$ is contained in the stable manifold of $a$ for $\hnu$. Now,
the stable manifold of~$a$ for~$\hnu$ is a Lagrangian submanifold, so it equals
$\sN E^-(a)$. 
\end{proof}

Before describing the dynamics of $\cnu$, we recall that a contact vector field
$\eta$ on a contact manifold $(V,\xi)$ has a special invariant manifold, called
its ``dividing hypersurface'', which is the set of points $a \in V$ where $\eta
(a) \in \xi_a$. This dividing hypersurface is empty if and only if $\eta$ is a
Reeb vector field, it contains all the singularities of $\eta$, and it is smooth
when these singularities are nondegenerate (see \cite{Gi}). For $\cnu$, the
dividing hypersurface is the sphere conormal bundle $\sSN(\nu)$ of $\nu$, which, 
by definition, is the union of the sphere conormal bundles of its orbits. It has
an obvious projection $\pi \on \sSN(\nu) \to M$ and can be viewed as a singular
sphere bundle over $M$ : it is a smooth $\S^{n-2}$--bundle over $M - \{\nu=0\}$
compactified by the fibers $\sST_aM \simeq \S^{n-1}$, $a \in \{\nu=0\}$. It can
alternatively be described as the projection to $\sST M$ of the zero-set of the
Hamiltonian function $g$ of $\hnu$ restricted to $\sT M - M$.

\begin{lemma}[properties of $\cnu$] \label{l:cnu}
Let $\nu$ be a vector field on $M$ with nondegenerate singularities, and let
$\cnu$ denote its contact lift on $\sST M$.

\stem{a}
The singularities of $\cnu$ in $\sST M$ lie in the fibers over the singularities
of $\nu$. For each zero $a$ of $\nu$, they form spheres in $\sST_aM$ along which
$\cnu$ is transversely nondegenerate. These are the spheres of the eigenspaces
of $-\Lnu_a^\top$ associated with the real eigenvalues. The non-real eigenvalues
give rise to invariant spheres filled up with periodic orbits along which $\cnu$
is transversely nondegenerate provided the eigenvalue is not pure imaginary.

\stem{b}
Let $a$ be a hyperbolic singularity of $\nu$ with $k$ attracting directions, and
denote by $C_a^-,C_a^+ \subset \sST_aM$ the respective projections of the stable
and unstable subspaces of $-\Lnu_a^\top = \hnu \rst{ \sT_aM }$. Then $C_a^-$ is
an invariant $(n-k-1)$--sphere which, inside the hypersurface $\sSN(\nu)$, is
transversely hyperbolic with stable manifold $\sSN E^-(a)$ and unstable manifold
$\pi^{-1}(E^+(a)) - \sSN E^+(a)$\tu; moreover, $\cnu$ is expanding along $C_a^-$
in the direction normal to $\sSN(\nu)$. Similarly, $C_a^+$ is an invariant
$(k-1)$--sphere which, inside $\sSN(\nu)$, is transversely hyperbolic with
unstable manifold $\sSN E^+(a)$ and stable manifold $\pi^{-1}(E^-(a)) - \sSN
E^-(a)$, and $\cnu$ is contracting along $C_a^+$ in the direction normal to
$\sSN(\nu)$.
\end{lemma}

\begin{proof}
Here again, everything is obvious except maybe the properties of the spheres
$C_a^-$ and $C_a^+$ over a hyperbolic singularity $a$ of $\nu$. First of all, we
note that the sphere conormal bundle of any invariant manifold for $\nu$ lies in
$\sSN(\nu)$. Hence, $\sSN(\nu)$ contains $\sSN E^-(a)$ and $\sSN E^+(a)$. Next,
$\cnu$ being the projection of $\hnu$ implies that $\sSN E^-(a)$ and $\pi^{-1}
\bigl(E^+(a)\bigr) - \sSN E^+(a)$ are respectively included in the stable and
unstable manifolds of $C_a^-$, and for dimensional reasons, they equal these
manifolds inside the invariant manifold $\sSN(\nu)$. Finally the behavior of
$\cnu$ in the direction normal to $\sSN(\nu)$ along $C_a^-$ follows also from
the behavior of $\hnu$ on $\sT_aM$ along $P_a^-$ by projection to $\sST_aM$.
\end{proof}

In the next sections, we shall apply the above considerations to an adapted
gradient $\nu$ of a given Morse function $\varphi \on M \to \R$. In this case,
it is useful to note that, for any regular value $u$ of $\varphi$, the inverse
image $\pi^{-1}(Q_u) \subset \sSN(\nu)$ of the level set $Q_u \de \{\varphi=u\}$
can be canonically identified with $\sST Q_u$: at a point of $Q_u$, hyperplanes
of $\rT Q_u$ correspond one-to-one to the hyperplanes of $\rT M$ which contain
$\nu$. The dividing hypersurface $\sSN(\nu)$ is actually a major character in
our story since, as we will see at the end of the paper, it is the double of the
Lefschetz fiber $F$ of $(\varphi,\nu)$. Another important remark is that, as a
consequence of the Morse--Smale property of $\nu$, the sphere conormal bundles
of the stable and unstable manifolds of $\nu$ are disjoint in $\sST M$ (see the
proof of Proposition \ref{p:f}).
 
\section{Fibrations with prescribed imaginary part}

The next three sections are devoted to the proof of Theorems \ref{t:closed} and
\ref{t:compact}. As in the statements of those results, $\varphi \on M \to \R$
is a Morse function and $\nu$ an adapted gradient. We choose a Riemannian metric
on $M$ for which $\nu = \nabla\varphi$ and which is the Euclidean metric in some
Morse coordinates near each critical point (such a metric is easily constructed
with a partition of unity). We denote by $\rho \on \sT M \to \R$ the associated
kinetic energy:
\[ \rho(p,q) \de \tfrac12 |p|^2 \quad
   \text{for all $q \in M$, $p \in \sT_qM$.} \]
For every $r>0$, we also set:
\begin{align*}
	W_r &\de \{(p,q) \in \sT M \with |p| \le r\} \subset \sT M, \\
	V_r &\de \{(p,q) \in \sT M \with |p| = r\} \simeq \sST M.
\end{align*}
Finally, we define $\Crit \subset M$ and $\Discr \subset \R$ to be the critical
locus of $\varphi$ and its discriminant locus, respectively, and, to avoid
irrelevant complications, we systematically assume that no two critical points
have equal values: $\varphi$ induces a bijection $\Crit \to \Discr$. 

\begin{proposition}[fibration criterion] \label{p:g}
Let $M$ be a connected manifold, $\varphi \on M \to \R$ a Morse function, and
$\nu$ an adapted gradient of $\varphi$. Assume that the Hamiltonian lift $\hnu$
of $\nu$ admits a Lyapunov function $f \on \sT M \to \R$ which extends $\varphi$
and is $1$--homogeneous near infinity. As usual, $g$ is the function $g(p,q)
\de \prg{p,\nu(q)}$.

\stem{a}
If $M$ is closed then the map $h \de f+ig \on \sT M \to \C$ is a symplectic
Lefschetz fibration.

\stem{b}
If $M$ is compact with non-empty boundary, suppose in addition that $\pm\dvf
\lambda \cdot f \ge 0$ on $\del^\pm\sT M$. Then the map $h \de f+ig \on \sT M
\to \C$ has the following properties, where $B \de \varphi(M) \oplus i\R \subset
\C$\tu:
\begin{itemize}
\item
$h \rst{ h^{-1}(B) } \on h^{-1}(B) \to B$ is a symplectic Lefschetz
fibration\tu;
\item
$\sT M$ retracts onto $h^{-1}(B)$ along the orbits of $\hnu$.
\end{itemize}
\end{proposition}

This statement does not explicitly require $\nu$ to satisfy the Morse--Smale
condition because, as explained in Remark \ref{r:MS} below, this property is a
consequence of $\hnu$ admitting a Lyapunov function which is homogeneous near
infinity.

\begin{proof}
As we already explained, $f$ being a Lyapunov function for the Hamiltonian field
$\hnu$ of $g$ ensures that every fiber $F_w = h^{-1}(w)$ of $h$ is a symplectic
submanifold of $(\sT M, \omega)$ away from the critical points of $h$, and since
$\omega$ is exact, the induced symplectic form is exact. For every $w \in \C$,
we set $\lambda_w \de \lambda \rst{ F_w - \Crit }$, and we denote by $\eta$ the
vector field on $\sT M - \Crit$ equal to the Liouville field $\dvf\lambda_w$ on
each fiber $F_w$.

Now we have two completeness issues to address: the completeness of the vector
field $\eta$ and the completeness of the (singular) connection $\zeta$ which is
the symplectic orthogonal complement of the distribution $\Ker \dh$. For every
$t \in \R_{>0}$, we write $\sigma_t \on \sT M \to \sT M$ the fiberwise dilation 
by $t$. Then the relations $\sigma_t^*\lambda = t\lambda$ and $\sigma_t^*h = th$
near infinity imply that $\eta$ is invariant by every $\sigma_t$, $t>0$, and
hence is complete. On the other hand, $\sigma_t^*(\drho/\rho) = \drho/\rho$ on
$\sT M - M$, so there exists a constant $C>0$ such that 
\[ \drho_{(p,q)} (v) \le C \rho(p,q) \, \bigl|\dh_{(p,q)}(v)\bigr| \quad
	\text{for all $(p,q) \in \sT M$ and $v \in \zeta_{(p,q)}$.} \]
Then the completeness of $\zeta$ follows from the divergence of the integral
$\int_1^\infty \dx / x$; indeed, this divergence shows that any horizontal curve
in $\sT M$ along which $\rho$ goes to infinity is mapped by $h$ to a path of 
infinite length in $\C$.

We will prove that, for all sufficiently large $r \ge r_w$, the intersection
$F_w \cap V_r$ is a positive contact submanifold of $V_r$. If $M$ is closed,
this implies that $F_w$ is a Liouville manifold for every $w \in \C - \Discr$.
If $M$ has boundary, this conclusion remains valid for $w \in B = (\varphi(M)
\oplus i\R) - \Discr$ due to the behavior of $\varphi$ on $\del M$ and to the
assumption that $\pm (\dvf\lambda \cdot f) \ge 0$ on $\del^\pm \sT M$: this
shows that $F_w \cap \del\sT M$, even though it may be non-empty, consists of
points where $F_w$ remains a smooth submanifold of $\sT M$ and has no boundary
there.

\begin{claim}
The function $f$ vanishes near infinity.
\end{claim}

\begin{proof}[Proof of the claim]
Choose $r>0$ large enough that $f$ is $1$--homogeneous outside $W_r$, meaning
that $\dvf\lambda \cdot f = f$.

If $M$ has boundary and $\del^-M, \del^+M$ are both non-empty, this homogeneity
condition, together with the condition $\pm(\dvf\lambda) \cdot f) \rst{ \del^\pm
\sT M } \ge 0$, implies that $\pm f \rst{ \del^\pm\sT M - W_r } \ge 0$. Hence,
$f$ vanishes in $\sT M - W_r$ since $M$ is connected.

If $M$ is closed, suppose (arguing by contradiction) that $f$ is positive on
$\sT M - W_r$. The homogeneity condition $\dvf\lambda \cdot f = f$ then implies
that the level sets of $f$ are transverse to $\dvf\lambda$ in $\sT M - W_r$. Now
choose a level set $X_s \de \{f=s\}$ with $s>0$ so large that $X_s$ is contained
in $\sT M - W_r$ (and is therefore diffeomorphic to $\sST M$). The condition
$\hnu \cdot f > 0$ then says that $\hnu$ is pointing transversely upward along
$X_s$, which is obviously impossible for a Hamiltonian vector field.

If $M$ has boundary but either $\del^-M$ or $\del^+M$ is empty then a mix of the
two previous arguments yields the result. Assume for instance that $\del^+M$ is
not empty but $\del^-M$ is. Then $f \ge 0$ on $\del^+\sT M - W_r$. If $f$ does
not vanish, we can construct a level set $X_s = \{f=s\}$ (diffeomorphic to $\sST
M$) which encloses a compact region of $\sT M$. On the boundary of this region
(made up of $X_s$ and a piece of $\del^+\sT M$), the vector field $\hnu$ points
transversely outward, which is again impossible.
\end{proof}

\begin{claim}
If $r>0$ is so large that $\dvf\lambda \cdot f = f$ near $V_r$ then $F_0 \cap
V_r$ is a non-empty positive contact submanifold of $V_r$.
\end{claim}

\begin{proof}[Proof of the claim]
Since $f$ and $g$ are homogeneous, the Liouville field $\dvf\lambda$ is tangent
to $F_0$ near $V_r$. On the other hand, $\dvf\lambda$ points transversely
outward along $V_r$, so it also points transversely out of $F_0 \cap W_r$ along
$F_0 \cap V_r$. This means exactly that $F_0 \cap V_r$ is a 
(non-empty) positive contact submanifold of $V_r$.
\end{proof}

To complete the proof of Proposition \ref{p:g}, we fix an $r_0>0$ such
that $f$ is homogeneous in $\sT M - W_{r_0}$. Since $F_0 \cap V_{r_0}$ is a
contact submanifold of $V_{r_0}$ and this property is ``open'', there exists an
$\eps>0$ such that, for all $|u|, |v| \le \eps$, the intersection $F_w \cap
V_{r_0}$, $w = u+iv$, is a positive contact submanifold of $V_{r_0}$.

Now pick any number $s>0$ and set $r(s) = r_0 s / \eps$. We claim that, for all
$|u|, |v| \le s$ and $r \ge r(s)$, the intersection $F_w \cap V_r$ is a positive
contact submanifold of $V_r$. Indeed, the radial projection $V_r \to V_{r_0}$
takes $F_w \cap V_r$ to $F_{r_0w/r} \cap V_{r_0}$ (due to the homogeneity of $f$
and $g$), and the hypotheses $r \ge r(s) = r_0 s / \eps$ and $|u|, |v| \le s$
imply that $|r_0u/r|, |r_0v/r| \le \eps$, and so $F_{r_0w/r} \cap V_{r_0}$ is a
positive contact submanifold of $V_{r_0}$.
\end{proof}

\begin{remark}[on the Morse--Smale condition] \label{r:MS}
We briefly explain here why the Morse--Smale condition is necessary in Theorems
\ref{t:closed} and \ref{t:compact}.

Assume that $\nu$ violates the Morse--Smale condition. This means that $\nu$ has
an orbit $\gamma$, running from a critical point $a$ to a critical point $b$,
along which the unstable manifold $E^+(a)$ and the stable manifold $E^-(b)$ of
$\nu$ are not transverse. Therefore, given a point $c \in \gamma$, the subspace
$\rT_cE^+(a) + \rT_cE^-(b)$ lies in some hyperplane $\tau_c \subset \rT_cM$.
Spreading $\tau_c$ by the flow of $\nu$, we get a hyperplane field $\tau$ along
$\gamma$ which contains $\rT E^+(a) + \rT E^-(b)$ at every point of $\gamma$ and
which extends up to the endpoints $a$ and $b$ of $\gamma$ (this can be seen from
the shape of $\nu$ in Morse coordinates near $a$ and $b$). We then consider,
over the segment $C \de \gamma \cup \{a,b\}$, the real line bundle
\[ R \de \bigcup_{q \in C} \tau_q^\bot \subset \sT M \rst C. \]
We denote by $\del_aR$ and $\del_bR$ the components of $\del R$ containing $a$
and $b$, respectively. By construction (see also Lemma \ref{l:hnu}), $\hnu$ is
tangent to $R$ as well as to $\del R$, and the dynamics of $\hnu \rst R$ is very
simple:
\begin{itemize}
\item
$\hnu \rst{ \del_aR }$ (resp.\ $\hnu \rst{ \del_bR }$) is a dilating (resp.\ 
contracting) linear vector field on a real line;
\item
all orbits of $\hnu$ in $R - \del R$ go from $a$ to $b$.
\end{itemize}
Viewing $C$ as the zero-section of $R$, we choose one of the two components of
$R - C$ and denote its closure by $R^+ \subset R$. We also define $\del_q^+R = 
\del_qR \cap R^+$ for $q \in \{a,b\}$.

If there exists a Lefschetz fibration $h$ of the form $h = f+ig$, then $\hnu$ is
a pseudogradient of $f$, and it follows from the dynamical behavior of $\hnu$
on $R$ that:
\begin{itemize}
\item
$f(R)$ is contained in $I \de [\varphi(a),\varphi(b)]$ (observe that this bound
already prevents $f$ from being homogeneous of degree $1$ or more); 
\item
$f(\del_a^+R)$ (resp.\ $f(\del_b^+R)$) is an open interval of $I$ containing $a$
(resp.\ $b$), and these two open intervals are disjoint. For the latter claim,
we argue as follows: if $f(a,p_0) = f(b,p_1) = u$, then near $(a,p_0)$ (resp.\ 
$(b,p_1)$) the set $\{f=u\} \cap R$ is an arc transverse to $\del_a^+R$ (resp.\ 
$\del_b^+R$); hence, there are plenty of interior orbits which intersect both
arcs, so $f$ takes the same value twice on every such orbit, contradicting that
$\hnu$ is a pseudogradient of $f$.
\end{itemize}
Now take $u \de f(a,p)$ for some $(a,p) \in \del_a^+R$ with $p \ne 0$. The above
observations show that, for the symplectic connection defined by $h$ (which is
spanned by $\hnu$ over $\R$), the arc $[u,\varphi(b)]$ has no horizontal lift
stemming from $(a,p) \in \del_a^+R$. Hence, the connection is not complete, so
$h$ is not a fibration.

To summarize this discussion, if $\nu$ violates the Morse--Smale condition, then
its Hamiltonian lift $\hnu$ provides a necessarily incomplete connection between
the level sets of any Lyapunov function it admits.
\end{remark}

\section{Coarse complexification of a Morse function} \label{s:coarse}

In the same framework as before, we denote by $\hnu$ and $\cnu$ the Hamiltonian
and contact lifts of $\nu$, respectively, and we use the splitting
\[ \sT M - M = \sST M \times \R_{>0}, \quad \text{with} \quad
	\sST M \times \{r\} = V_r \quad \text{for all $r>0$,} \]
to view $\cnu$ as a vector field on $\sT M - M$ tangent to each hypersurface
$V_r$.

We recall that our overall goal is to extend the Morse function $\varphi$ to a
symplectic Lefschetz fibration $h \on \sT M \to \C$, and in this section we
construct $h$ in a neighborhood of the zero-section $M \subset \sT M$. Very
explicitly, we consider the map $h^0 \on \sT M \to \C$ defined by 
\[ h^0(p,q) \de \varphi(q) - \tfrac12 \chi(q) \nabla^2\varphi_q (p,p)
	+ i \prg{p, \nu(q)}, \quad \text{for all $(p,q) \in \sT M$.} \]
Here $\nabla^2\varphi_q$ is the covariant second derivative of $\varphi$ at $q$,
regarded as a symmetric pairing on $\sT_qM \cong \rT_qM$, and $\chi \on M \to
[0,1]$ is a cut-off function equal to $1$ at least in a neighborhood of $\Crit$;
if $M$ is closed, $\chi \equiv 1$ is perfectly fine, but if $M$ has boundary it
is technically convenient to take $\chi \equiv 0$ near $\del M$. As a matter of
fact, $h^0$ is a ``first order complexification'' of $\varphi$. The imaginary
part $g \de \im h^0$ is the function whose Hamiltonian field is $-\hnu$ (cf.\
Section \ref{s:lifts}) and it will remain globally the imaginary part of our
final Lefschetz fibration $h$. As for the real part $f^0 \de \re h^0$, we will
have to modify it far away from the zero-section (we will mostly rearrange its
critical values) but it has nice basic properties near the zero-section:  

\begin{lemma}[properties of $h^0$] \label{l:h0}
Let $M$ be a compact manifold, $\varphi \on M \to \R$ a Morse function, and
$\nu$ an adapted gradient of $\varphi$. Then there exists a radius $\delta>0$
such that\tu:
\begin{enumerate}
\itemup{1}
in $W_\delta$, the critical points of $h^0$ coincide with those of $\varphi$ and
are of complex Morse type\tu;
\itemup{2}
in $W_\delta$, the real part $f^0 = \re h^0$ is a Lyapunov function for the 
Hamiltonian field~$\hnu$\tu;
\itemup{3}
on $V_r$ with $0 < r \le \delta$, the restriction $f_r^0 \de f^0 \rst{ V_r }$ is
a Morse--Bott Lyapunov function for the contact field $\cnu$.
\end{enumerate}
\end{lemma}

The geometric meaning of property 2 is that, in $W_\delta$, the fibers of $h^0$
are symplectic submanifolds away from $\Crit$. Indeed, since $-\hnu$ is the
Hamiltonian field of $g = \im h^0$, the condition $\hnu \cdot f^0 \ne 0$ implies
that $\df^0$ and $\dg$ are independent and that $\Ker \dh^0 = \Ker \df^0 \cap
\Ker \dg$ is a hyperplane of $\Ker \dg$ transverse to $\hnu$, hence a symplectic
subspace.

On the other hand, it follows from property 3 that the critical submanifolds
of each function $f_r^0$, $0 < r \le \delta$, are the components of the zero-set
of $\cnu$. Near each critical point $a$ of $\varphi$, since the chosen metric is
the standard Euclidean metric in Morse coordinates $(q_1 \etc q_n)$ centered on
$a$, we have
\[ \nu(q) = \nabla\varphi(q) = - 2q_1 \del_{q_1} -\dots- 2q_k \del_{q_k} 
	+ 2q_{k+1} \del_{q_{k+1}} +\dots+ 2q_n \del_{q_n}. \]
Thus, all positive (resp.\ negative) eigenvalues of $\Lnu_a$ are equal, and the
components of the zero-set of $\cnu$ in $\sST_aM$ are the two spheres $C_a^\pm$ 
defined in Lemma \ref{l:cnu}-b. The corresponding critical values of $f_r^0$ are 
$f_r^0(C_a^\pm) = \varphi(a) \pm 2r^2$, $0 < r \le \delta$.

\begin{proof}[Proof of the lemma]
The reason for adding $-\frac12 \nabla^2\varphi$ to $\varphi$ in the real part
$f^0$ of $h^0$ is that, at every critical point $a$ of $\varphi$, the gradient
of the quadratic form $-\frac12 \nabla^2\varphi_a$ on $\sT_aM$ for the Euclidean
metric is the linear vector field $-\Lnu_a^\top = \hnu \rst{ \sT_aM }$. Then, by
orthogonal projection to any sphere $S_{a,r} \de \sT_aM \cap V_r$, the gradient
of the Morse--Bott function $f^0 \rst{ S_{a,r} }$ for the round metric  is $\cnu
\rst{ S_{a,r} }$.

Furthermore, the function $(p,q) \mapsto \nabla^2\varphi_q (p,p)$ vanishes
identically together with its $1$--jet along the zero-section $M \subset \sT M$,
so $f^0$ and (the pullback of) $\varphi$ are arbitrarily $\cC^1$--close near $M$.
As a result, for any open subset $U$ of $M$ containing $\Crit$, the positivity
of $\nu \cdot \varphi$ on the compact set $M-U$ implies that $\hnu \cdot f^0$ 
and $\cnu \cdot f^0$ (where $\cnu$ is viewed as a vector field on $\sT M - M$
tangent to the hypersurfaces $V_r$) are both positive on $W_{\delta_0} \cap \sT
(M-U)$ for some $\delta_0 > 0$.

To complete the proof, we study $h^0$ more carefully near the critical points of
$\varphi$. For $a \in \Crit$, let $(q_1 \etc q_n)$ be coordinates centered at
$a$, on a neighborhood $U_a$, in which the metric is Euclidean and
\[ \varphi(q) = \varphi(0) + \sum_1^n \sgn_j q_j^2, \quad
   \sgn_j \in \{-1,+1\}. \]
In the associated cotangent coordinates, we have:
\begin{align*}
   \nu(q) &= 2 \sum_1^n \sgn_j q_j \del_{q_j}, \\
   \nabla^2\varphi_q(p,p) &= 2 \sum_1^n \sgn_j p_j^2, \\
   \hnu(p,q) &= 2 \sum_1^n \sgn_j (q_j \del_{q_j} - p_j \del_{p_j}), \\
   g(p,q) &= 2 \sum_1^n \sgn_j p_jq_j.
\end{align*}
Thus, in the complex coordinates $z_j = p_j+iq_j$ on $\sT U_a$, $1 \le j \le n$,
the map $h^0$ is given by
\[ h^0(z) = \varphi(0) + \sum_1^n \sgn_j z_j^2. \]
Hence, $a$ is the only critical point of $h^0$ in $\sT U_a$ and it is of complex
Morse type, which proves 1. Furthermore,
\[ (\hnu \cdot f^0)(z) = 4 \sum_1^n |z_j|^2, \]
which proves 2.

Finally, the contact lift $\cnu$, viewed again as a vector field on $\sT M - M$,
has the form $\cnu = \hnu - \ups \dvf\lambda$, where the function $\ups$ can be
computed by evaluating the differential of the kinetic energy $\rho$ (which is
zero on $\cnu$). In $\sT U_a - U_a$, we have $\rho = \frac12 \sum_1^n p_j^2$,
and so 
\[ \ups(p,q) = - \frac{ \sum_1^n \sgn_j p_j^2 }{ \sum_1^n p_j^2 }. \]
As a result, again in $\sT U_a - U_a$,
\begin{align*}
	(\cnu \cdot f^0) (p,q) &= 4 \sum_1^n q_j^2 + 4 \frac
	{ \lrp{ \sum_1^n p_j^2 }^2 - \lrp{ \sum_1^n \sgn_j p_j^2 }^2 }
   { \sum_1^n p_j^2 } \\
	&= 4 \sum_1^n q_j^2 + 4 \frac
   { \lrp{\sum_1^n (1+\sgn_j) p_j^2} \cdot \lrp{\sum_1^n (1-\sgn_j) p_j^2} }
   { \sum_1^n p_j^2 }.
\end{align*}
This function is non-negative and vanishes exactly (and identically) along the
spheres $C_a^\pm$. The transverse non-degeneracy of these critical submanifolds
is clear in the $q$--directions and, in the $p$--directions (lying in the fiber
$\sT_aM$), it follows from the observation that every non-degenerate quadratic
form on a Euclidean space restricts to a Morse--Bott function on any sphere.
This concludes the proof of 3.
\end{proof}

\begin{remark}[properties of $\ups$] \label{r:upsilon}
We quickly work out here a few properties of the function $\ups$ which we will
need to prove Proposition \ref{p:f}. For any point $a \in \Crit$, in the same
complex coordinates $z_j = p_j+iq_j$ as before, $1 \le j \le n$, the spheres
$C_a^\pm \subset \sST_aM$ are given by 
\[ C_a^\pm = \{p \in \S^{n-1} \simeq \sST_aM \with
	p_j=0\ \text{if}\ \sgn_j = \pm1\} \subset \sST_aM = \{q=0\}. \] 
Then the explicit expression of $\ups$ yields $\ups \rst{ C_a^\pm } = \pm2$. 
Furthermore, a little calculation (similar to the previous ones) shows that, on
$\sST_aM$,
\[ (\cnu \cdot \ups) (z) = 8 \frac
   { \lrp{\sum_1^n (1+\sgn_j) p_j^2} \cdot \lrp{\sum_1^n (1-\sgn_j) p_j^2} }
	{ \lrp{ \sum_1^n p_j^2 }^2 }. \]
This function is clearly non-negative and vanishes exactly on the zero-set of 
$\cnu \rst{ \sST_aM }$, namely $C_a^- \cup C_a^+$. 
\end{remark}

\section{Rearrangement of critical values}

We recall that a Morse function $f$ is said to be \emph{ordered} if the order of
critical values is consistent with that of indices: $\ind(a) < \ind(b)$ implies
$f(a) < f(b)$ for any two critical points $a, b$. It is well-known that, if the
gradient of $f$ (for some metric) satisfies the Morse--Smale condition, then
one can deform $f$ among Morse functions with the same gradient (for different
metrics) to an ordered Morse function \cite[Theorem B]{Sm}. This rearrangement
process, applied to the Morse--Bott functions $f_r^0 \on V_r \simeq \sST M \to
\R$ of Lemma \ref{l:h0} (with $r \le \delta$), is the key trick we need for
constructing our symplectic Lefschetz fibration.

\begin{proposition}[rearrangement process] \label{p:f}
Let $M$ be a compact manifold, $\varphi \on M \to \R$ a Morse function, and
$\nu$ an adapted gradient of $\varphi$ satisfying the Morse--Smale condition.
Then the contact vector field $\cnu$ admits a family of Lyapunov Morse--Bott
functions $f_r \on \sST M \to \R$ $(r>0)$ with the following properties\tu:
\begin{enumerate}
\itemup{1}
for $r$ sufficiently small, $f_r = f_r^0 \on V_r \cong \sST M \to \R$\tu;
\itemup{2}
for $r$ sufficiently large, the function $f_\infty \de f_r/r \on \sST M \to \R$
is independent of $r$ and vanishes transversely\tu;
\itemup{3}
for any $r>0$, the function $\pm\del_rf_r$ is non-negative near $\del^\pm\sST M$
while it is positive on $C_a^\pm$ for every $a \in \Crit$\tu;
\itemup{4}
the function $f \on \sT M \to \R$ given by $f \rst{ V_r } \de f_r$ and $f \rst M
\de \varphi$ is a smooth Lyapunov function for the Hamiltonian field $\hnu$\tu;
\itemup{5}
the function $f \on \sT M \to \R$ is invariant under the fiberwise antipodal
involution.
\end{enumerate}
\end{proposition}

\begin{proof}
The key remark is that, for any $a, b \in \Crit$, no trajectories of $\cnu$ go
from $C_a^+$ to $C_b^-$. Indeed, by Lemma \ref{l:cnu}, the stable manifold of
$C_b^-$ for $\cnu$ is the sphere conormal bundle of the stable manifold $E^-(b)$
of $b$ for $\nu$ while the unstable manifold of $C_a^+$ for $\cnu$ is the sphere
conormal bundle of the unstable manifold $E^+(a)$ of $a$ for $\nu$. But since
$\nu$ satisfies the Morse--Smale condition, the submanifolds $E^+(a)$ and $E^-
(b)$ are transverse to each other, and hence their sphere conormal bundles do
not intersect. The next lemma is a direct consequence of this remark. We define
\[ A^\pm \de \del^\pm \sST M \cup \bigcup_{a \in \Crit} E^\pm(C_a^\pm) \]
where $E^-$ and $E^+$ denote here the stable and unstable manifolds for $\cnu$. 
We also recall that the vector fields $\hnu$ and $\cnu$ on $\sT M - M$ satisfy a
relation of the form 
\[ \cnu = \hnu - \ups \dvf\lambda, \]
where $\ups = \hnu \cdot \log r$ is a function independent of $r$ (namely, the
pullback of some function on $\sST M$).

\begin{lemma}[slope function] \label{l:slope}
The contact vector field $\cnu$ admits a Lyapunov function $f_\infty \on \sST
M \to \R$ such that\tu:
\begin{enumerate}
\itemup{i}
$f_\infty$ is negative on $A^-$ and positive on $A^+$\tu; 
\itemup{ii}
$(\cnu \cdot f_\infty) + \ups f_\infty > 0$ everywhere on $\sST M$\tu;
\itemup{iii}
$\ups f_\infty \ge 0$ in $\sST U$ for some neighborhood $U$ of $\Crit$ in $M$.
\end{enumerate}
\end{lemma}

\begin{proof}[Proof of the lemma]
Since the sets $A^\pm$ defined above are disjoint they possess (small, and hence
disjoint) neighborhoods $H^\pm$ whose boundaries are smooth hypersurfaces 
transverse to $\cnu$ and which retract onto $A^\pm$ along the flow lines of
$\cnu$. Concretely, $H^\pm$ is obtained from a collar neighborhood of $\del^\pm
\sST M$ by successive attachments of Morse--Bott handles, following the order
of critical values for $H^-$ and the inverse order for $H^+$. Then $\del H^\pm$
contains $\del^\pm \sST M$, and we set
\[ K^\pm \de \del H^\pm - \del^\pm \sST M. \]
With these notations, the closure of $\sST M - (H^- \cup H^+)$ is a cobordism in
which all orbits of $\cnu$ go from $K^-$ to $K^+$, so it is a product. Let $K
\subset \sST M - (H^- \cup H^+)$ be a closed hypersurface transverse to $\cnu$.
For any $a \in \Crit$, every orbit of $\cnu$ in $\sT_aM - (C_a^- \cup C_a^+)$
meets $K$ in one point and contains also a unique zero of the function $\ups$,
with $\cnu \cdot \ups > 0$ at that point (see Remark \ref{r:upsilon}). Hence we
can slide $K$ along the orbits of $\cnu$ to make it agree with $\{\ups=0\}$ in
$\sST U$ for some neighborhood $U$ of $\Crit$ in $M$. Now, using the handlebody
structure of $H^\pm$, it is not hard to construct a function $f_\infty \on
\sST M \to \R$ with the following properties:
\begin{itemize}
\item
$f_\infty$ is a Lyapunov function for $\cnu$ and its zero set is $K$, so it is
negative on the component $V^-$ of $\sST M - K$ containing $H^-$ and positive on
the other component $V^+$ containing $H^+$; 
\item
$f_\infty$ is so increasing along $\cnu$ that $\cnu \cdot \log|f_\infty| >
-\ups$ in $V^+$ and $\cnu \cdot \log|f_\infty| < -\ups$ in $V^-$.
\end{itemize}
The construction of a function satisfying the first property is standard in
Morse theory \cite{Sm}, and its generalization to the Morse--Bott case is
straightforward. The second property is a minor add-on to the first. It holds
trivially on the zero-set of $\cnu$ because $\ups$ is negative on the spheres
$C_a^-$ and positive on the spheres $C_a^+$; it can be achieved globally because
the function $\log|f_\infty|$ is not bounded below (it tends to $-\infty$ on
both sides of~$K$). In practice, the germ of $f_\infty$ along $K$ can be chosen
\emph{a priori} with $f_\infty \rst K = 0$ and $(\cnu \cdot f_\infty) \rst K >
0$, and its extensions over $V^-$ and $V^+$ can then be performed independently.
The properties stated in the lemma are direct consequences of the above ones. In
particular, since $\ups$ and $f_\infty$ have the same sign at every point of
$\sST U$, their product is non-negative.
\end{proof}

We now return to the proof of the proposition. Before defining the family $f_r$,
we note that the function $f^1 \on \sT M - M \to \R$ given by $f^1 \rst{ V_r } 
= rf_\infty$ satisfies
\[ \hnu \cdot f^1 =
	r \bigl((\cnu \cdot f_\infty) + \ups f_\infty\bigr) > 0. \]
The family $f_r$ will take the shape
\[ f_r \de \tau_0(r)\, f_r^0 + \tau_1(r)\, f_r^1, \]
where $f_r^1 = rf_\infty$ and $\tau_0, \tau_1 \on \R_{>0} \to \R_{\ge0}$ have
the following properties:
\begin{itemize}
\item
$\tau_0(r) = 1$ for $r \le \delta/2$ and $\tau_0(r) = 0$ for $r \ge \delta$;
\item
$\tau'_0(r) \le 0$ and $|\tau'_0(r)| \le 3/\delta$ for all $r>0$; 
\item
$\tau_1(r) = 0$ for $r \le \eps/2$ and $\tau_1(r) = C$ for $r \ge \eps$, with
$C \gg 0$ and $0 < \eps \ll \delta/2$;
\item
$\tau'_1(r) \ge 0$ and $|\tau'_1(r)| < 3C / \eps$ for all $r>0$.
\end{itemize}
(In particular, the derivatives $\tau'_0$ and $\tau'_1$ have disjoint supports.) 

The functions $f_r$ are clearly Lyapunov functions for $\cnu$, and they satisfy
properties 1 and 2. As for property 3, a simple calculation shows that it
holds provided $C$ (the value of $\tau_1(r)$ for $r \ge \eps$) is sufficiently
large. To check property 4 --- saying that the function $f \on \sT M \to \R$
defined by $f \rst{ V_r } = f_r$ is a Lyapunov function for $\hnu$---, we write
\[ \hnu \cdot f =
	\tau_0(r)\, (\hnu \cdot f^0) + \tau_1(r)\, (\hnu \cdot f^1)  
	+ r \tau'_0(r)\, \ups f_r^0 + r^2 \tau'_1(r)\, \ups f_\infty. \]
In the domain $\{r \ge \delta\}$, this quantitiy is positive because $f$ equals
$f^1$. In the region $\{\delta/2 \le r \le \delta\}$, the function $\tau_1$ is
constant equal to $C$, and
\[ C\, (\hnu \cdot f^1) + r \tau'_0(r)\, \ups f_r^0 
	= r \left(C\, \bigl((\cnu \cdot f_\infty) + \ups f_\infty\bigr) 
	+ \tau'_0(rs)\, \ups f_r^0\right) \]
is positive as soon as $C$ is sufficiently large. Finally, in the region $\{r
\le \delta/2\}$, the function $\tau_0$ is constant equal to $1$, so
\[ \hnu \cdot f = (\hnu \cdot f^0)
	+ \tau_1(r)\, (\hnu \cdot f^1) + r^2 \tau'_1(r)\, \ups f_\infty. \]
In the neighborhood $\sST U$ of the cotangent fibers over critical points, this
derivative is positive because the first two terms of the right-hand side are
positive while the third term is non-negative (see Lemma \ref{l:slope}). Outside
$\sST U$, the first term ($\hnu \cdot f^0$) is bounded below by a constant
$\kappa>0$ and the second term ($\tau_1(r) (\hnu \cdot f^1)$) is non-negative.
As for the last term ($r^2\tau'_1(r) \ups f_\infty$), since $r^2 |\tau'_1(r)|
\le 3C\eps$, it can be made smaller than $\kappa$ by taking $\eps$ sufficiently
small. 

It remains to prove property 5, \emph{i.e.}, that $f$ can be chosen invariant 
under the action of the fiberwise antipodal involution. First we note that the
vector field $\cnu$ and the functions $f^0_r$, $r \le \delta$, are invariant. To
arrange that all functions $f_r$ are invariant, we apply the same construction
as above but we work on the projective cotangent bundle instead of the sphere
cotangent bundle, and then we lift the functions we obtain to $\sST M$.
\end{proof}

\begin{proof}[Proofs of Theorems \ref{t:closed} and \ref{t:compact}]
We apply Proposition \ref{p:g} with the function $f$ provided by Proposition
\ref{p:f}.
\end{proof} 

\section{The Lefschetz fiber} \label{s:lf}

This section collects various informations about the topology and the symplectic
geometry of the Lefschetz fiber of the upgraded function $(\varphi,\nu)$, namely
the regular fiber of our symplectic Lefschetz fibration extending $\varphi$.   

\begin{proposition}[the Weinstein structures of real fibers] \label{p:w}
Let $M$ be a compact manifold, $\varphi \on M \to \R$ a Morse function, $\nu$ an
adapted gradient of $\varphi$ satisfying the Morse--Smale condition, and $h \on
\sT M \to M$ the symplectic Lefschetz fibration given by Propositions \tref{p:g}
and \tref{p:f}. Then every fiber $F_u = h^{-1}(u)$, $u \in \R - \Discr$, has a
Weinstein structure induced by the canonical $1$--form $\lambda$ of $\sT M$, a
Lyapunov function for the corresponding Liouville field being the Morse--Bott
function $\rho_u \de \rho \rst{ F_u }$, where $\rho$ is the kinetic energy of
the underlying metric. Moreover, these Weinstein structures belong to the same
homotopy class. 
\end{proposition}  

The last assertion of this proposition should be understood as follows: given any
embedded arc $I$ in $\C - \Discr$ joining two real regular values $u_0, u_1$,
the symplectic fibers $F_w$, $w \in I$, admit Weinstein structures which, for
$w \in \del I = \{u_0,u_1\}$, are the Weinstein structures induced by $\lambda$.

\begin{proof}
To check that $\lambda_u \de \lambda \rst{ F_u }$ is a Weinstein structure on
$F_u$, we first show that $\rho_u$ is a Morse--Bott function and then prove that
each level set $\{\rho_u = r^2/2\}$ is a positive contact submanifold of $V_r$
at any point which is non-critical for $\rho_u$.

Let $\sN(\nu)$ denote the hypersurface $\sN(\nu) \de \{g=0\} \subset \sT M$,
which is smooth away from $\Crit$ (the critical points of $g$ are the zeros of
its Hamiltonian field $\hnu$). The regular fiber $F_u$ and the zero-section $M$
both sit in $\sN(\nu)$, and inside $\sN(\nu)$, they intersect transversely along
the level set $Q_u = \{\varphi=u\}$. Hence $F_u \cap M$ is a transversely
non-degenerate minimum submanifold of $\rho_u$.

Set $Y \de \sN(\nu) - M$. Then $f \rst Y$ and $\rho \rst Y$ have no critical
points, the former because $\df(\hnu) > 0$ while $\dg(\hnu) = 0$, and the latter
because $\drho(\dvf\lambda) > 0$ while $\dg(\dvf\lambda) = 0$ along $Y$. It then
follows from the Lagrange multiplier theorem that the critical points of $\rho$
on the fibers of $h$ with real values (namely, the critical points of $\rho \rst
Y$ on the level sets of $f \rst Y$) coincide with the critical points of the
functions $f_r \rst Y$, $r>0$ (which are the critical points of  $f \rst Y$ on
the level sets of $\rho \rst Y$). Moreover, since $f_r$ is a Lyapunov function
for $\cnu$ and $Y$ is an invariant submanifold containing all zeros of $\cnu$,
the critical points of $f_r \rst Y$ are just the critical points of $f_r$, and
so they form transversely non-degenerate critical submanifolds. Thus, $\rho_u$ is
a Morse--Bott function.

\begin{remark}[critical submanifolds of $\rho_u$]
It is very instructive to precisely spot the critical submanifolds of $\rho_u$.
By Proposition \ref{p:f}, properties 2 and 3, for every point $a \in \Crit$,
the critical value $f_r(C_a^+)$ (resp.\ $f_r(C_a^-)$) is an increasing (resp.\
decreasing) and unbounded function of $r$. Therefore, given $u \in \R - \Discr$
and $a \in \Crit$, there exists a unique $r$ such that $f_r(C_a^+) = u$ (resp.\ 
$f_r(C_a^-) = u$) if $\varphi(a) < u$ (resp.\ $\varphi(a) > u$) and none
otherwise. The critical submanifolds of $\rho_u$, besides $Q_u \de F_u \cap M$,
are the spheres $C_a^\pm$ sitting in the corresponding $V_r$. This provides many
different presentations of the regular fiber as a Weinstein manifold.
\end{remark}

Now consider a point $(p,q) \in \{\rho_u = r^2/2\} = F_u \cap V_r$. If $(p,q)$
is non-critical for $\rho_u$, the above discussion ensures that $\cnu$ does not
vanish at $(p,q)$, and hence $(\cnu \cdot f_r) (p,q) > 0$. We then observe that
$\cnu$ spans the (contact geometric) characteristic foliation of $Y \cap V_r$ in
$V_r$ (indeed, $Y \cap V_r$ is the dividing hypersurface of the contact vector
field $\cnu$ on $V_r$). Therefore, the inequality $(\cnu \cdot f_r) (p,q) > 0$
implies that $F_u \cap V_r$, which is contained in $Y \cap V_r$, is transverse
to the characteristic foliation of $Y \cap V_r$ at $(p,q)$, and so is a contact
submanifold of $V_r$ at $(p,q)$.

It remains to show that the Weinstein structures of the real fibers lie in a
common homotopy class. Let $u_0, u_1 \in \R - \Discr$. If $u_0, u_1$ are in the
same component of $\R - \Discr$, clearly the Weinstein structures of $F_{u_0}$
and $F_{u_1}$ are homotopic through those of the regular fibers $F_u$, $u \in
[u_0,u_1]$. Therefore (assuming as usual that no two critical points have the
same value), it suffices to treat the case where there is exactly one point $a
\in \Crit$ such that $u_0 < \varphi(a) < u_1$. To simplify the notations, we set
$\varphi(a) = 0$. Since $u_0, u_1$ can be chosen arbitrarily small, the homotopy
only requires a local construction near $a$; we now describe the process.

The map $h$ resulting from Propositions \ref{p:g} and \ref{p:f} provides local
complex Darboux coordinates $(z_1 \etc z_n) \in \C^n$ centered on $a$ in which
$h$ is the model function $\sum_1^n z_j^2$ while $M$ is represented by $M_k$,
$\varphi$ by $\varphi_k$, and $\rho$ by $\rho_k$ where $k \de \ind_\varphi(a)$
(see Example~\ref{x:locmod} for the notations). We use these coordinates to
identify a neighborhood $D$ of $a$ with the ball $\{z \in \C^n \with |z| \le 3
\delta\}$ for some $\delta>0$. Next, we choose a function $\chi \on [0,3\delta]
\to [0,1]$ that equals $1$ on $[0,\delta]$ and $0$ on $[2\delta,3\delta]$. Then,
for any $w=u+iv \in \C$ with $v$ sufficiently small (with respect to $\delta$),
we let $F'_w$ denote the submanifold of $\sT M$ obtained as follows:
\begin{itemize}
\item
$F'_w \cap D$ is defined by the equation 
\[ \sum_1^n z_j^2 = u + i v \chi(|z|)\,; \]
\item
$F'_w - D \de F_u - D$.
\end{itemize}
Thus, $F'_u = F_u$ if $u$ is real. For $u>0$ sufficiently small and for $w \de
u e^{(1-t)i\pi}$, $t \in [0,1]$, the manifolds $F'_w$ connect $F_{-u}$ to $F_u$,
and we would expect the $1$--forms induced on them by $\lambda$ to provide the
desired homotopy of Weinstein structures. This is roughly correct, but a slight
perturbation of $\lambda$ is necessary beforehand.

\begin{lemma}[perturbation of $\lambda$] \label{l:sharp}
There exists a $1$--form $\lambda^\#$ on $\sT M$ such that\tu:
\begin{itemize}
\item
$\lambda^\#$ coincides with $\lambda$ outside $D$;
\item
$\lambda^\#$, inside $D$ viewed as a ball in $\C^n$, equals $\dc\rho^\#$ where
$\rho^\#$ is pseudoconvex and arbitrarily $\cC^2$ close to $\rho$\tu;
\item
for some $\eps>0$ depending on $\rho^\#$, the $1$--form induced by $\lambda^\#$
on each $F_w \cap D$, $|w| \le \eps$, is non-singular in $F_w \cap E$ where $E
\de \{\delta \le |z| \le 2\delta\}$.  
\end{itemize}
\end{lemma}

With this lemma, we can complete the proof of the proposition as planned. First,
inside $D$, we write $\lambda = \lambda_k = \dc\rho_k = \dc\rho$ (see Example
\ref{x:locmod}-b). Then we invoke the lemma, taking $\rho^\# - \rho$ so small
that $\rho^\#$ is pseudoconvex in $D$. Thus, on each intersection $F_w \cap D$,
the form $\lambda^\#$ induces a Liouville form $\lambda^\#_w$ whose Liouville
field $\dvf\lambda^\#_w$ is gradientlike for the restriction of $\rho^\#$.
Moreover,
according to the third point of the lemma, $\dvf\lambda^\#_w$ is non-singular in
the compact region $F_w \cap E$.

Take $w = u+iv \in \C$ with $|w| \le \eps$ and pick any point $z \in F'_w \cap
E$. Then $z$ also belongs to $F_{w'} \cap E$ where $w' = u + iv \chi(|z|)$ and
the tangent spaces of $F_{w'}$ and $F'_w$ at $z$ are respectively defined by  
\begin{align*}
   0 &= \sum_1^n z_j \dz_j, \\
   0 &= \sum_1^n z_j \dz_j - iv \chi'(|z|) \d(|z|).
\end{align*}
For $\eps$ sufficiently small, these spaces are so close (uniformly in $z \in E
\cap \bigcup_{|w| \le \eps} F'_w$) that the $1$--form $\lambda'_w$ induced by 
$\lambda^\#$ on each $F'_w$ with $|w| \le \eps$ is a non-singular Liouville
form whose Liouville field $\dvf\lambda'_w$ is gradientlike for the restriction  
of $\rho^\#$. As a result, each $(F'_w, \lambda'_w)$ is a Weinstein manifold.
Hence, for $u \in (0,\eps]$ and for $w = u e^{(1-t)i\pi}$, $t \in [0,1]$, the
Weinstein manifolds $(F'_w, \lambda'_w)$ define a homotopy from $(F_{-u},
\lambda'_{-u})$ to $(F_u, \lambda'_u)$. Finally, we concatenate this homotopy at
each end of the interval with the barycentric homotopy from $\lambda_{\pm u}$ to
$\lambda'_{\pm u}$, which consists of Weinstein structures on $F_{\pm u}$, to
obtain the desired homotopy from $(F_{-u}, \lambda_{-u})$ to $(F_u, \lambda_u)$.
Note that the symplectic form $\dlambda'_w$ on the perturbed fiber $F'_w$ does
not coincide in $D$ with the one induced by the symplectic form of $\sT M$, but
if necessary, this can be remedied using Moser's trick.    
\end{proof}
   
\begin{proof}[Proof of the lemma]
Let $\sigma \on \R_{\ge0} \to \R_{\ge0}$ be a function satisfying $\sigma(t) =
t$ for $t \le 4\delta^2$ and $\sigma(t) \equiv 0$ for $t \ge 9\delta^2$.
We define $\rho^\#_k \on \C^n \to \R$ by
\[ \rho^\#_k(z) \de \rho_k(z) + \tfrac c2 \sigma(|z|^2), \]
where $c$ is a real parameter which we fix small enough that $\rho^\#$ is
pseudoconvex. 

On the one hand, the tangent space of $F_w$ at any point $z = x+iy$ is the
complex vector space defined by 
\[ \sum_1^n z_j \dz_j = 0. \]
On the other hand,
\[ \drho^\#_k = \drho_k
 + \tfrac c2 \sigma'(|z|^2)\, \d(|z|^2). \]
A routine calculation then shows that $\rT_zF_w$ is contained in the kernel of
$\lambda^\# = \dc\rho^\#$ if and only if $z = (z_1 \etc z_n)$, where $z_j = x_j
+ iy_j$, satisfies the following equations:
\begin{align*}
   y_lx_j &= x_ly_j && \text{for $0 \le j < l \le k$,} \tag{1} \\
   y_lx_j &= x_ly_j && \text{for $k < j < l \le n$,} \tag{2} \\
   x_lx_j &= y_ly_j && \text{for $0 \le j \le k < l \le n$,} \tag{3} \\ 
   [1+c\sigma'(|z|^2)]\, y_lx_j &= c \sigma'(|z|^2)\, x_ly_j &&
   \text{for $0 \le j \le k < l \le n$.} \tag{4}
\end{align*}
Let $K$ be the set of those points $z \in E$ which are solutions of the above
system. Since $K$ is compact, we just need to show that $h$ does not vanish on
$K$, which is easily checked by a case-by-case analysis that we briefly sketch
below.

For $z \in K$, we first note that $\sigma'(|z|^2) = 1$. Then we write $z = (z',
z'') \in \C^k \times \C^{n-k}$, with $z' = x'+iy'$ and $z'' = x''+iy''$. By (1),
$x' = (x_1 \etc x_k)$ and $y' = (y_1 \etc y_k)$ are linearly dependent; by (2),
$x'' = (x_{k+1} \etc x_n)$ and $y'' = (y_{k+1} \etc y_n)$ are also linearly
dependent. Now assume for instance that $x' \ne 0$ and $y'' \ne 0$. Then $y' =
\mu' x'$ and $x'' = \mu'' y''$ for some $\mu', \mu'' \in \R$. Next, (3) implies
that $\mu'' = \mu' \mathrel{=:} \mu$, and it follows from (4) that $\mu^2 =
(1+c)/c \ne 0$. As a result, 
\[ \im h(z) = 2\mu (|x'|^2 + |y''|^2)
   = \frac{ 2\mu }{ 1+\mu^2 } |z|^2 \ne 0. \]
The other cases are pretty similar.
\end{proof}

We now complete this paper with a few more remarks and comments about the
topology and the geometry of the Lefschetz fiber we have constructed.

Fix $u \in \R - \Discr$. The Morse--Bott handle decomposition of $F_u$ given by
$\rho_u$ can be viewed as follows. As we already said, the level set $Q_u$ is
the minimum of $\rho_u$. To detect the other critical submanifolds, we appeal
to Proposition \ref{p:f}: by properties 2 and 3, for every point $a \in \Crit$,
the value $f_r(C_a^+)$ (resp.\ $f_r(C_a^-)$) is an increasing (resp.~decreasing)
and unbounded function of~$r$; hence, for each $a \in \Crit$, there is a unique
$r$ such that $f_r(C_a^+) = u$ (resp.\ $f_r(C_a^-) = u$) if $\varphi(a) < u$
(resp.~$\varphi(a) > u$), and none otherwise. By the Lagrange multiplier theorem
(as seen in the proof above), the critical submanifolds of $\rho_u$, besides
$Q_u$, are the spheres $C_a^\pm$ contained in $F_u$, and they lie in $V_r$ where
$f_r(C_a^\pm) = u$. Thus, $\rho_u$ has exactly one critical sphere for each $a
\in \Crit$, which is $C_a^+$ (of dimension $\ind(a)-1$) if $\varphi(a) < u$ and
$C_a^-$ (of dimension $n-\ind(a)-1$) if $\varphi(a) > u$.

Now take two critical points $a, b$ with consecutive values in $\Discr$. If $u
< \varphi(a) < \varphi(b)$, then $\rho_u(C_a^-) < \rho_u(C_b^-)$ (for the copies
of $C_a^-, C_b^-$ sitting in $F_u$); but if $\varphi(a) < u < \varphi(b)$, the
ordering of $\rho_u(C_a+)$ and $\rho_u(C_b^-)$ depends on the position of $u$
between $\varphi(a)$ and $\varphi(b)$ (and for some $u$, the two spheres lie at
the same level of $\rho_u$).

In the situation where $\varphi(a) < u < \varphi(b)$, we can also locate $Q_u$
with respect to the vanishing cycles 
$Z_u(a), Z_u(b) \subset F_u$ associated with $a, b$. The invariant manifolds $E^+(a)$ and $E^-(b)$ being transverse to
each other, their sphere conormal bundles are disjoint. Moreover, it follows
from the discussion in Example \ref{x:locmod}-c that $Q_u$ intersects cleanly
$Z_u(a)$ and $Z_u(b)$ along the attaching spheres $K_u(a) +Q_u \cap E^+(a)$ and
$K_u(b) = Q_u \cap E^-(b)$, respectively. The vanishing cycle $Z_u(a)$ is the
union of $K_u(a)$ and the stable manifold of $C_a^+$ for the gradientlike field
$\dvf\lambda_u$, and similarly for $Z_u(b)$. Finally, Lemma \ref{l:surgery}
implies that, for any real regular value $u_-$ (resp.\ $u_+$) in the component
of $\R - \Discr$ immediately below $\varphi(a)$ (resp.\ above $\varphi(b)$), the
parallel transport $F_{u_\pm} \to F_u$ takes $Q_{u_\pm}$ to an exact Lagrangian
submanifold isotopic to the Mores--Boot lagrangian surgery of $Q_u$ with
$Z_u(a)$ and $Z_u(b)$, respectively. All these incidence relations, of course,
are preserved by parallel transport.      

\begin{example}[Lefschetz fiber of a Heegaard splitting]
Let $M$ be a closed oriented $3$--manifold, $\varphi \on M \to \R$ an ordered
Morse function with only one minimum and one maximum, $\nu$ an adapted gradient
satisfying the Morse--Smale condition, and $Q$ a regular level set of $\varphi$
separating the critical points of index $1$ from those of index $2$. The above
discussion shows that the Lefschetz fiber of $(\varphi,\nu)$ is the Weinstein
$4$--manifold $F$ obtained as follows from the Heegaard splitting given by~$Q$.

Let $g$ denote the genus of $Q$. Then the unstable (resp.\ stable) manifolds of
the critical points of index $1$ (resp.\ $2$) intersect $Q$ along $g$ disjoint
embedded curves $\alpha_1 \etc \alpha_g$ (resp.\ $\beta_1 \etc \beta_g$), and
since $\nu$ satisfies the Morse--Smale condition, each $\alpha_j$ is transverse
to each $\beta_l$. This transversality implies that the sphere conormal bundles
\[ \sSN\alpha_1 \etc \sSN\alpha_g, \sSN\beta_1 \etc \sSN\beta_g
   \subset \sSN Q \]
are disjoint, and hence provide $4g$ disjoint embedded framed curves in the
boundary of the disk cotangent bundle $\sDT Q$. Then $F$ is the (completion of)
the Weinstein domain obtained by attaching a Weinstein handle to $\sDT Q$ along
each of these $4g$~curves. The vanishing cycle associated to any critical point     
of index $1$ (resp.\ $2$) is the union of the corresponding $\sN\alpha_j$ (resp.\
$\sN\beta_l$) and the two Weinstein handles attached to it. The vanishing cycle
associated with the minimum (resp.\ the maximum), or equivalently the level set
just above the minimum (resp.\ below the maximum) appears as the result of the
$g$ Morse--Bott Lagrangian surgeries of $Q$ with the vanishing cycles including
the curves $\alpha_j$ (resp.\ $\beta_l$). We refer to \cite{Sr} for details.
\end{example}

A more global picture of the Lefschetz fiber is provided by the hypersurface
$\sSN(\nu)$ of $\sST M$ which, as already mentioned, is its double. To see this,
pick a real value $u \in \R - \varphi(M)$ which is so large that $F_u$ lies
entirely in the region of $\sT M$ where $h$ is homogeneous --- that is, where
$f_r = rf_\infty$ in the notations of Proposition \ref{p:f}. Then the projection
$F_u \subset \sT M - M \to \sST M$ maps $F_u$ diffeomorphically to one half of
$\sSN(\nu)$ (determined by the position of $u$ with respect to $\varphi(M)$)
limited by the level set $f_\infty = 0$. The proof is roughly as follows: first
of all, everything takes place in $\sN(\nu) = \{g=0\}$, and $\dvf\lambda$ is
tangent to $\sN(\nu)$; next, $rf_\infty = u$ is equivalent to $f_\infty = u/r$;
this shows that $\dvf\lambda$ is transverse to $F_u$ inside $\sN(\nu)$, and so
it projects diffeomorphically to its image in $\sSN(\nu)$. Moreover, this
diffeomorphism maps the Liouville field $\dvf\lambda_u$ to a vector field which
is proportional to $\cnu$ (because at each point, $\dvf\lambda$, $\hnu$ and
$\cnu$ lie in the same tangent plane).  

\medskip

We conclude this discussion with a few words about the relationships between our
Lefschetz fibration and the Weinstein structure of the cotangent bundle $\sT M$.  
Symplectic Lefschetz fibrations on a Weinstein manifold $W$ are often requested
to satisfy more conditions than those given in Definition \ref{d:slf} (see for
instance the notion of ``abstract Weinstein Lefschetz fibration'' in \cite{GP}).
Mainly, the Lefschetz thimbles associated to a complete system of vanishing
paths (see \cite{Se} are required to appear also as the top-dimensional handles
in a Weinstein presentation of $W$. Our Lefschetz fibration $h \on W = \sT M \to \C$
has this property (for any Morse function $\varphi \on M \to \R$), so maybe it is
worth explaining why briefly. 

Take a smooth arc $I$ in $\C$ which avoids $\Discr$ and intersects each bounded
component of $\R - \Discr$ transversely in a single point. Then a sufficiently
small metric neighborhood of $I$ is a smooth disk $D \subset \C - \Discr$ which
intersects each bounded component of $\R - \Discr$ in a segment, and we can find
disjoint segments $J_u \subset \R$, $u \in \Discr$, such that each $J_u$ avoids
$D$ in its interior but connects $u$ to a point in $\del D$. By Proposition \ref
{p:w}, we can endow each fiber $F_w = h^{-1}(w)$, $w \in D$, with a Weinstein
structure. Then $h^{-1}(D)$, equipped with the vanishing cycles provided in its
boundary by the segments $J_u$, $u \in \Discr$, is roughly what is called an
abstract Weinstein Lefschetz fibration in \cite{GP}. In other words, $h$ can be
regarded as a (hopefully a bit more concrete at this point) Weinstein Lefschetz
fibration.

\end{document}